\UseRawInputEncoding

\documentclass[oneside, 11pt]{amsart} 

\usepackage{amsmath,amsthm,amssymb,epic}
\usepackage{eqlist,eqparbox}
\usepackage{amsfonts}
\usepackage{latexsym}
\usepackage[2emode]{psfrag}
\usepackage{amsthm}
\usepackage{amsmath}
\usepackage[all]{xy}

\newcommand{\Title}[1]{\bigskip\bigskip\centerline{\bf #1}\bigskip}
\newcommand{\Author}[1]{\medskip\centerline{ \it #1}}

\newcommand{\Affiliation}[1]{\medskip\centerline{#1}}
\newcommand{\Email}[1]{\medskip\centerline{#1}\bigskip}

\begin{document}

\newcommand{\N}{\mbox {$\mathbb N $}}
\newcommand{\Z}{\mbox {$\mathbb Z $}}
\newcommand{\Q}{\mbox {$\mathbb Q $}}
\newcommand{\R}{\mbox {$\mathbb R $}}
\newcommand{\lo }{\longrightarrow }
\newcommand{\ul}{\underleftarrow }
\newcommand{\rl}{\underrightarrow }
\newcommand{\rs }{\rightsquigarrow }
\newcommand{\ra }{\rightarrow } 
\newcommand{\dd }{\rightsquigarrow } 
\newcommand{\rars }{\Leftrightarrow }
\newcommand{\ol }{\overline }
\newcommand{\la }{\langle }
\newcommand{\tr }{\triangle }
\newcommand{\xr }{\xrightarrow }
\newcommand{\de }{\delta }
\newcommand{\pa }{\partial }
\newcommand{\LR }{\Longleftrightarrow }
\newcommand{\Ri }{\Rightarrow }
\newcommand{\va }{\varphi }
\newcommand{\Den}{{\rm Den}\,}
\newcommand{\Ker}{{\rm Ker}\,}
\newcommand{\Reg}{{\rm Reg}\,}
\newcommand{\Fix}{{\rm Fix}\,}
\newcommand{\Sup}{{\rm Sup}\,}
\newcommand{\Inf}{{\rm Inf}\,}
\newcommand{\Img}{{\rm Im}\,}
\newcommand{\Id}{{\rm Id}\,}
\newcommand{\ord}{{\rm ord}\,}

\newtheorem{theorem}{Theorem}[section]
\newtheorem{lemma}[theorem]{Lemma}
\newtheorem{proposition}[theorem]{Proposition}
\newtheorem{corollary}[theorem]{Corollary}
\newtheorem{definition}[theorem]{Definition}
\newtheorem{example}[theorem]{Example}
\newtheorem{examples}[theorem]{Examples}
\newtheorem{xca}[theorem]{Exercise}
\theoremstyle{remark}
\newtheorem{remark}[theorem]{Remark}
\newtheorem{remarks}[theorem]{Remarks}
\numberwithin{equation}{section}

\def\leftmark{L.C. Ciungu}

\Title{CLASSIFICATION OF STATES ON CERTAIN ORTHOMODULAR STRUCTURES} 
\title[Classification of states on certain orthomodular structures]{}
                                                                       
\Author{\textbf{LAVINIA CORINA CIUNGU}}
\Affiliation{Department of Mathematics} 
\Affiliation{St Francis College}
\Affiliation{179 Livingston Street, Brooklyn, NY 11201, USA}
\Email{lciungu@sfc.edu}

\begin{abstract} 
We define various type of states on implicative involutive BE algebras (Jauch-Piron state, (P)-state, (B)-state, 
subadditive state, valuation), and we investigate the relationships between these states. 
Moreover, we introduce the unital, full and rich sets of states, and we prove certain properties involving these notions. In the case when an implicative involutive BE algebra possesses a rich or a full set of states, we prove that it is an implicative-orthomodular lattice. If an implicative involutive BE algebra possesses a rich set of (P)-states or a full set of valuations, then it is an implicative-Boolean algebra. 
Additionally, based on their deductive systems, we give characterizations of implicative-orthomodular lattices and 
implicative-Boolean algebras. \\

\textbf{AMS classification (2020):} 06C15, 03G25, 06A06, 81P15
\end{abstract}

\maketitle

\section{Introduction} 

Orthomodular lattices were first introduced by G. Birkhoff and J. von Neumann (\cite{Birk1}), and, independently, by 
K. Husimi (\cite{Husimi}) by studying the structure of the lattice of projection operators on a Hilbert space.  
They introduced the orthomodular lattices as the event structure describing quantum mechanical experiments. 
Later, orthomodular lattices and orthomodular posets were considered as the standard quantum logic \cite{Varad1} (see also \cite{Beran, DvPu}). 
Complete studies on orthomodular lattices are presented in the monographs \cite{Birk2} and \cite{Beran}. 
An orthomodular lattice (OML for short) is an ortholattice $(X,\wedge,\vee,^{'},0,1)$ verifying: \\
$(OM)$ $(x\wedge y)\vee ((x\wedge y)^{'}\wedge x)=x$, for all $x,y\in X$, or, equivalently \\ 
$(OM^{'})$ $x\vee (x^{'}\wedge y)=y$, whenever $x\le y$ (where $x\le y$ iff $x=x\wedge y$) (\cite{PadRud}). \\
The orthomodular lattices were redefined and studied by A. Iorgulescu (\cite{Ior32}) based on involutive m-BE algebras, 
and two generalizations of orthomodular lattices were given: orthomodular softlattices and orthomodular widelattices. 
Based on implicative involutive BE algebras, we redefined in \cite{Ciu83} the orthomodular lattices, by introducing 
and studying the implicative-orthomodular lattices (IOML for short), and we gave certain characterizations of these structures. The distributivity of implicative-orthomodular lattices was investigated in \cite{Ciu84}. \\
In this paper, we study in the case of implicative-orthomodular lattices certain notions and results presented by 
S. Pulmannov\'a in \cite{Pulm1} for orthomodular lattices. 
We define various type of states on implicative involutive BE algebras (Jauch-Piron state, (P)-state, (B)-state, 
subadditive state, valuation), and we study the relationships between these states. 
Moreover, we introduce the unital, full and rich sets of states, and we investigate certain properties involving 
these notions. In the case when an implicative involutive BE algebra $X$ possesses a rich or a full set of states, 
we prove that $X$ is an implicative-orthomodular lattice. 
If $X$ possesses a rich set of (P)-states or a full set of valuations, then $X$ is an implicative-Boolean algebra. 
Additionally, we define the notions of deductive systems, order deductive systems, q-deductive systems, and 
p-deductive systems of implicative involutive BE algebras.  
Based on their deductive systems, we give characterizations of implicative-orthomodular lattices and 
implicative-Boolean algebras.

$\vspace*{1mm}$

\section{Preliminaries}

We recall some basic notions and results regarding the implicative BE algebras and implicative-orthomodular 
lattices that will be used in the paper. 

\emph{BE algebras} were introduced in \cite{Kim1} as algebras $(X,\ra,1)$ of type $(2,0)$ satisfying the 
following conditions, for all $x,y,z\in X$: 
$(BE_1)$ $x\ra x=1;$ 
$(BE_2)$ $x\ra 1=1;$ 
$(BE_3)$ $1\ra x=x;$ 
$(BE_4)$ $x\ra (y\ra z)=y\ra (x\ra z)$. 
A relation $\le$ is defined on $X$ by $x\le y$ iff $x\ra y=1$. 
A BE algebra $X$ is \emph{bounded} if there exists $0\in X$ such that $0\le x$, for all $x\in X$. 
In a bounded BE algebra $(X,\ra,0,1)$ we define $x^*=x\ra 0$, for all $x\in X$. 
A bounded BE algebra $X$ is called \emph{involutive} if $x^{**}=x$, for any $x\in X$. 

\begin{lemma} \label{qbe-10} $\rm($\cite{Ciu83}$\rm)$ 
Let $(X,\ra,1)$ be a BE algebra. The following hold for all $x,y,z\in X$: \\
$(1)$ $x\ra (y\ra x)=1;$ 
$(2)$ $x\le (x\ra y)\ra y$. \\
If $X$ is bounded, then: \\
$(3)$ $x\ra y^*=y\ra x^*;$ 
$(4)$ $x\le x^{**}$. \\
If $X$ is involutive, then: \\
$(5)$ $x^*\ra y=y^*\ra x;$ 
$(6)$ $x^*\ra y^*=y\ra x;$ 
$(7)$ $(x\ra y)^*\ra z=x\ra (y^*\ra z);$ \\
$(8)$ $x\ra (y\ra z)=(x\ra y^*)^*\ra z;$    
$(9)$ $(x^*\ra y)^*\ra (x^*\ra y)=(x^*\ra x)^*\ra (y^*\ra y)$.  
\end{lemma}

\noindent
In a BE algebra $X$, we define the additional operation: \\
$\hspace*{3cm}$ $x\Cup y=(x\ra y)\ra y$. \\
If $X$ is involutive, we define the operation: \\
$\hspace*{3cm}$ $x\Cap y=((x^*\ra y^*)\ra y^*)^*$, \\
and the relations $\le_Q$ and $\le_L$ by: \\
$\hspace*{3cm}$ $x\le_Q y$ iff $x=x\Cap y$ and $x\le_L y$ iff $x=(x\ra y^*)^*$. \\
\noindent
We can see that $x\le_L y$ iff $x^*=y\ra x^*$. 

\begin{proposition} \label{qbe-20} $\rm($\cite{Ciu83}$\rm)$ Let $X$ be an involutive BE algebra. 
Then the following hold for all $x,y,z\in X$: \\
$(1)$ $x\le_Q y$ implies $x=y\Cap x$ and $y=x\Cup y;$ \\
$(2)$ $\le_Q$ is reflexive and antisymmetric; \\
$(3)$ $x\Cap y=(x^*\Cup y^*)^*$ and $x\Cup y=(x^*\Cap y^*)^*;$ \\ 
$(4)$ $x\le_Q y$ implies $x\le y;$ \\
$(5)$ $x, y\le_Q z$ and $z\ra x=z\ra y$ imply $x=y$ \emph{(cancellation law)}; \\
$(6)$ $\le_L$ implies $\le$; \\ 
$(7)$ $\le_L$ is an order relation on $X$. 
\end{proposition}

\noindent 
A \emph{quantum-Wajsberg algebra} (\emph{QW algebra, for short}) (\cite{Ciu83}) $(X,\ra,^*,1)$ is an 
involutive BE algebra $(X,\ra,^*,1)$ satisfying the following condition: for all $x,y,z\in X$, \\
(QW) $x\ra ((x\Cap y)\Cap (z\Cap x))=(x\ra y)\Cap (x\ra z)$. \\
It was proved in \cite{Ciu83} that condition (QW) is equivalent to the following conditions: \\
$(QW_1)$ $x\ra (x\Cap y)=x\ra y;$ \\ 
$(QW_2)$ $x\ra (y\Cap (z\Cap x))=(x\ra y)\Cap (x\ra z)$. \\
An involutive-BE algebra satisfying condition $(QW_2)$ is called an \emph{implicative-orthomolular algebra}. 
A BE algebra $X$ is called \emph{implicative} if it satisfies the following condition, for all $x,y\in X$: \\
$(Impl)$ $(x\ra y)\ra x=x$. \\
 
\begin{lemma} \label{iol-30} \emph{(\cite{Ciu83})} Let $(X,\ra,^*,1)$ be an implicative involutive BE algebra. 
Then $X$ verifies the following axioms: for all $x,y\in X$, \\
$(iG)$       $x^*\ra x=x$, or equivalently, $x\ra x^*=x^*;$ \\
$(Iabs$-$i)$ $(x\ra (x\ra y))\ra x=x;$ \\
$(Pimpl)$    $x\ra (x\ra y)=x\ra y$.    
\end{lemma}

\begin{lemma} \label{iol-30-10} \emph{(\cite{Ciu83})}
Let $(X,\ra,^*,1)$ be an involutive BE algebra. The following are equivalent: \\
$(a)$ $X$ is implicative; \\
$(b)$ $X$ verifies axioms $(iG)$ and $(Iabs$-$i)$; \\ 
$(c)$ $X$ verifies axioms $(Pimpl)$ and $(Iabs$-$i)$. 
\end{lemma}

\begin{lemma} \label{iol-30-20} Let $X$ be an implicative involutive BE algebra. The following hold, 
for all $x,y\in X$: \\ 
$(1)$ $x\le_L y$ iff $y^*\le_L x^*;$ \\
$(2)$ $\le_Q$ implies $\le_L;$ \\
$(3)$ $x\le_L y\ra x$ and $x^*\le_L x\ra y$.   
\end{lemma} 
\begin{proof}
$(1)$,$(2)$ See \cite[Lemma 3.4]{Ciu84}. \\
$(3)$ Applying $(Impl)$, we have $(x\ra (y\ra x)^*)^*=((y\ra x)\ra x^*)^*=((x^*\ra y^*)\ra x^*)^*=x$, 
hence $x\le_L y\ra x$. Similarly, since $(x^*\ra (x\ra y)^*)^*=((x\ra y)\ra x)^*=x^*$, we get $x^*\le_L x\ra y$.  
\end{proof}

\begin{lemma} \label{ioml-30} \emph{(\cite{Ciu83})} Let $X$ be an involutive BE algebra. 
The following are equivalent for all $x,y\in X$: \\
$(IOM)$ $x\Cap (y\ra x)=x;$ \\
$(IOM^{'})$ $x\Cap (x^*\ra y)=x;$ $\hspace*{5cm}$ \\
$(IOM^{''})$ $x\Cup (x\ra y)^*=x$. 
\end{lemma}

\begin{definition} \label{ioml-30-10} \emph{(\cite{Ciu83})} 
\emph{
An \emph{implicative-orthomodular lattice} (IOML for short) is an implicative involutive BE algebra satisfying 
any one of the equivalent conditions from Lemma \ref{ioml-30}.
}
\end{definition}

The orthomodular lattices $(X,\wedge,\vee,^{'},0,1)$ are term-equivalent to implicative-orthomodular 
lattices $(X,\ra,^*,1)$, by the mutually inverse transformations: \\
$\hspace*{3cm}$ $\varphi:$\hspace*{0.2cm}$ x\ra y:=(x\wedge y^{'})^{'}$ $\hspace*{0.1cm}$ and  
                $\hspace*{0.1cm}$ $\psi:$\hspace*{0.2cm}$ x\wedge y:=(x\ra y^*)^*$, \\
and the relation $x\vee y:=(x^{'}\wedge y^{'})^{'}=x^*\ra y$. 
The unary operation ``$^*$" is defined as $x^*:=x\ra 0=(x\wedge 0^{'})^{'}=x^{'}$. 

\begin{proposition} \label{ioml-40} \emph{(\cite{Ciu83})}
Let $X$ be an implicative-orthomodular lattice. Then the following hold for all $x,y,z\in X$: \\
$(1)$ $x\Cap (y\Cup x)=x$ and $x\Cup (y\Cap x)=x$. \\
If $x\le_Q y$, then: \\
$(2)$ $y\Cup x=y$ and $y^*\le_Q x^*;$ \\ 
$(3$ $y\ra z\le_Q x\ra z$ and $z\ra x\le_Q z\ra y;$ \\
$(4)$ $x\Cap z\le_Q y\Cap z$ and $x\Cup z\le_Q y\Cup z$. 
\end{proposition}

\begin{proposition} \label{ioml-50} \emph{(\cite{Ciu83})}
Let $X$ be an implicative-orthomodular lattice. Then the following hold for all $x,y,z\in X$: \\
$(1)$ $x\ra (y\Cap x)=x\ra y;$ \\  
$(2)$ $(x\Cup y)\ra (x\ra y)^*=y^*;$ \\ 
$(3)$ $x\Cap ((y\ra x)\Cap (z\ra x))=x;$ \\
$(4)$ $(x\ra y)\ra (y\Cap x)=x;$ \\
$(5)$ $\le_Q$ is an order relation on $X;$ \\
$(6)$ $x\le_Q y$ and $x\le y$ imply $x=y;$ \\
$(7)$ $x\Cap y\le_Q y\le_Q x\Cup y;$ \\
$(7)$ $(x\Cap y)\ra (y\Cap x)=1;$ \\
$(8)$ $(x\Cup y)\ra (y\Cup x)=1$.           
\end{proposition}

\begin{theorem} \label{ioml-90-10} \emph{(\cite{Ciu84})} 
Let $X$ be an implicative involutive BE algebra. The following are equivalent: \\
$(a)$ $X$ is an implicative-orthomodular lattice; \\
$(b)$ $X$ satisfies condition $(QW_1);$ \\ 
$(c)$ $X$ satisfies condition $(QW_2);$ \\ 
$(d)$ $X$ satisfies condition $(QW)$. 
\end{theorem}

It follows that any implicative-orthomodular lattice is a quantum-Wajsberg algebra. 

\begin{theorem} \label{dioml-05-50} \emph{(\cite{Ciu84})} Let $X$ be an implicative involutive BE algebra. 
The following are equivalent: \\
$(a)$ $X$ is an implicative-orthomodular lattice; \\
$(b)$ $\le_L$ implies $\le_Q;$ \\
$(c)$ for $x,y\in X$, $x\le_L y$ implies $y=y\Cup x$.  
\end{theorem}

As a consequence, if $X$ is an implicative-orthomodular lattice, then $\le_Q=\le_L$. \\
An implicative-orthomodular lattice $X$ is an \emph{implicative-Boolean algebra} if $x\Cup y=y\Cup x$, 
for all $x,y\in X$ (\cite{Ciu84}). We mention that the notion of implicative-Boolean algebras was 
first introduced by A. Iorgulescu in 2009 (\cite{Ior35}). 

\begin{definition} \label{cioml-10} \emph{(\cite{Ciu84})}
\emph{
Let $X$ be an implicative involutive BE algebra, and let $x,y,z\in X$. \\
$(1)$ We say that $x$ \emph{commutes} with $y$, denoted by $x\mathcal{C} y$, if $x=(x\ra y^*)\ra (x\ra y)^*$. \\
$(2)$ The triple $(x,y,z)$ is \emph{distributive} if it satisfies the following conditions: \\
$(Idis_1)$ $((x^*\ra y)\ra z^*)^*=(x\ra z^*)\ra (y\ra z^*)^*$, \\
$(Idis_2)$ $((x\ra y^*)\ra z)^*=((z^*\ra x)\ra (z^*\ra y)^*$. \\
$(3)$ $X$ is distributive if all triples $(x,y,z)$ are distributive. 
} 
\end{definition}

\begin{proposition} \label{cioml-20} \emph{(\cite{Ciu84})}
Let $X$ be an implicative involutive BE algebra and let $x,y\in X$. Then: \\
$(1)$ $x\mathcal{C} x$, $x\mathcal{C} 0$, $0\mathcal{C} x$, $x\mathcal{C} 1$, $1\mathcal{C} x$, $x\mathcal{C} x^*$, 
$x^*\mathcal{C} x;$ \\ 
$(2)$ $x\le_L y$ or $x\le_L y^*$ implies $x\mathcal{C} y;$ \\
$(3)$ $x\le_Q y$ or $x\le_Q y^*$ implies $x\mathcal{C} y;$ \\
$(4)$ $x\mathcal{C} (y\ra x);$ \\
$(5)$ if $X$ is an IOML, then $x\mathcal{C} y$ implies $x\mathcal{C} y^*$, $x^*\mathcal{C} y$, $x^*\mathcal{C} y^*$. 
\end{proposition}

\begin{theorem} \label{cioml-30} \emph{(\cite{Ciu84})} \emph{(Foulis-Holland theorem)} 
Let $X$ be an implicative-orthomodular lattice and let $x,y,z\in X$, such that 
one of these elements commutes with the other two. Then $(x,y,z)$ is a distributive triple. 
\end{theorem}

$\vspace*{1mm}$

\section{On the deductive systems of implicative-orthomodular lattices} 

In this section, we redefine the ideals and filters presented in \cite{DvPu} in the case of quantum-MV algebras, by 
introducing the notions of deductive systems, order deductive systems, q-deductive systems, and 
p-deductive systems of implicative involutive BE algebras.  
Based on the deductive systems, we give characterizations of implicative-orthomodular lattices and 
implicative-Boolean algebras. 

\begin{definition} \label{dsioml-10} \emph{
Let $X$ be an implicative involutive BE algebra, and let $F\subseteq X$ be a nonempty subset of $X$. 
Then $F$ is called: \\
$(1)$ an \emph{order deductive system} (o-DS for short), if it satisfies the following condition:\\
$(F_1)$ $x\in F$, $y\in X$, $x\le_L y$ imply $y\in F;$ \\
$(2)$ a \emph{q-deductive system} (q-DS for short), if $F$ satisfies $(F_1)$ and \\
$(F_2)$ $x,y\in F$ implies $(x\ra y^*)^*\in F$. \\
$(3)$ a \emph{p-deductive system} (p-DS for short), if $F$ is a q-DS satisfying \\
$(F_3)$ $x\in F$, $y\in X$ imply $x\Cup y\in F$. 
}
\end{definition}

\begin{proposition} \label{dsioml-10-10} If $X$ is an implicative-orthomodular lattice, then $(F_1)$ is 
equivalent to condition: \\
$(F_1^{'})$ $x\in F$, $y\in X$ imply $y\ra x\in F$. 
\end{proposition}
\begin{proof} 
Assume that $F$ satisfies $(F_1)$, and let $x\in F$, $y\in X$. 
By Lemma \ref{iol-30-20}$(3)$, we have $x\le_L y\ra x$, and by $(F_1)$ we get $y\ra x\in F$, that is $(F_1^{'})$. 
Conversely, let $F$ satisfying $(F_1)$, and let $x\in F$, $y\in X$ such that $x\le_L y$. 
According to Theorem \ref{dioml-05-50} and applying $(F_1^{'})$, we have $y=y\Cup x=(y\ra x)\ra x\in F$, 
so that condition $(F_1)$ is verified. 
\end{proof}

\begin{definition} \label{dsioml-20} \emph{
Let $X$ be an implicative involutive BE algebra, and let $F\subseteq X$. 
Then $F$ is called a \emph{deductive system} (DS for short), if it satisfies the following conditions: \\
$(DS_1)$ $1\in F;$ \\
$(DS_2)$ $x, x\ra y\in F$ implies $y\in F$. 
}
\end{definition}

Denote by $\mathcal{DS}$, $\mathcal{DS}_o$, $\mathcal{DS}_q$, $\mathcal{DS}_p$ the set of all 
deductive systems, order-deductive systems, q-deductive systems and p-deductive systems of $X$, respectively. 
One can easily check that $\mathcal{DS}(X)=\mathcal{DS}_p(X)\subseteq \mathcal{DS}_q(X)$.   

\begin{proposition} \label{dsioml-30} 
Let $X$ be an implicative involutive BE algebra, and let $F\subseteq X$. The following are equivalent: \\
$(a)$ $F\in \mathcal{DS}(X);$ \\
$(b)$ $F$ is nonempty satisfying $(F_2)$ and \\
$(F_4)$ $x\in F$, $y\in X$, $x\le y$ imply $y\in F;$ \\
$(c)$ $F$ is nonempty satisfying $(F_2)$ and $(F_3)$. 
\end{proposition}
\begin{proof}
$(a)\Rightarrow (b)$: Let $F\in \mathcal{DS}(X)$, and let $x,y\in F$. 
By $(DS_1)$ we have $1\in F$, so that $F$ is nonempty. 
Applying Lemma \ref{qbe-10}$(8)$ for $z:=(x\ra y^*)^*$, we get 
$x\ra (y\ra (x\ra y^*)^*)=(x\ra y^*)^*\ra (x\ra y^*)^*=1\in F$. 
Since $x,y\in F$, it follows that $(x\ra y^*)^*\in F$, hence $F$ satisfies condition $(F_2)$. 
Consider $x\in F$, $y\in X$ such that $x\le y$, that is $x\ra y=1\in F$. 
Since $x\in F$, by $(DS_2)$, we get $y\in F$, that is $F$ satisfies condition $(F_4)$. \\
$(b)\Rightarrow (a)$: Let $F\subseteq X$ nonempty satisfying $(F_2)$ and $(F_4)$. 
Let $x,y\in X$ such that $x,x\ra y\in F$, so that, by $(F_2)$, $(x\ra (x\ra y)^*)^*\in F$. 
We have: \\
$\hspace*{1.50cm}$ $(x\ra (x\ra y)^*)^*\ra y=y^*\ra (x\ra (x\ra y)^*)=x\ra (y^*\ra (x\ra y)^*)$ \\
$\hspace*{5.00cm}$ $=x\ra ((x\ra y)\ra y)=(x\ra y)\ra (x\ra y)=1$. \\
Hence $(x\ra (x\ra y)^*)^*\le y$, and by $(F_4)$, we get $y\in F$. It follows that $F\in \mathcal{DS}(X)$. \\   
$(b)\Rightarrow (c)$: Let $x\in F$ and $y\in X$. 
We have $x\ra (x\Cup y)=x\ra ((x\ra y)\ra y)=(x\ra y)\ra (x\ra y)=1$, hence $x\le x\Cup y$. 
By $(F_4)$ we get $x\Cup y\in F$, that is $(F_3)$ is verified. \\
$(c)\Rightarrow (b)$: If $x\in F$, $y\in X$ such that $x\le y$, then by $(F_3)$, $y=1\ra y=(x\ra y)\ra y=x\Cup y\in F$. 
Hence $F$ satisfies condition $(F_4)$. 
\end{proof}

Let $X$ be an implicative involutive BE algebra, and let $F\in \mathcal{DS}_o(X)$. Consider the following properties: \\
$(P_1)$ $x\in F$, $y\in X$, $y\le_L x$ imply $x\Cup y\in F;$ \\
$(P_2)$ $x\in F$, $y\in X$ imply $x\Cup y\in F$. \\

We prove the following results using an idea from \cite{Pulm1}. 

\begin{theorem} \label{dsioml-40} Let $X$ be an implicative involutive BE algebra. Then: \\
$(1)$ $X$ is an implicative-orthomodular lattice iff every $F\in \mathcal{DS}_o(X)$ satisfies property $(P_1);$ \\
$(2)$ $X$ is an implicative-Boolean algebra iff every $F\in \mathcal{DS}_o(X)$ satisfies property $(P_2)$. 
\end{theorem}
\begin{proof}
$(1)$ Suppose $X$ is an IOML, $F\in \mathcal{DS}_o(X)$, and let $x\in F$, $y\in X$ such that $y\le_L x$. 
Since $X$ is an IOML, by Theorem \ref{dioml-05-50} we have $x\Cup y=x\in F$, hence $F$ satisfies $(P_1)$. 
Conversely, assume that every $F\in \mathcal{DS}_o(X)$ satisfies property $(P_1)$. 
Let $x,y\in X$ such that $y\le_L x$ and consider $F=[x,1]=\{z\in X\mid x\le_L z\}$. 
It follows that $F\in \mathcal{DS}_o(X)$ and $x\in F$, so that, by $(P_1)$, $x\Cup y\in F$, hence $x\le_L x\Cup y$. 
From $y\le_L x$ we get $x^*\le_L y^*$, so that $x^*\le y^*$, thus $x^*\ra y^*=1$. 
Moreover, $x^*\le_L y^*$ implies $x^*=(x^*\ra y)^*$, that is $x=x^*\ra y$. 
Since $x\le_L x\Cup y$, we have 
$x^*=x\ra (x\Cup y)^*=x\ra (x^*\Cap y^*)=x\ra ((x^*\ra y^*)\ra y^*)^*=x\ra (y^*)^*=x\ra y$. 
It follows that $x\Cup y=(x\ra y)\ra y=x^*\ra y=x$, hence, by Theorem \ref{dioml-05-50}, $X$ is an IOML. \\
$(2)$ Let $X$ be an implicative-Boolean algebra, let $F\in \mathcal{DS}_o(X)$, and let $x\in F$, $y\in X$. 
Then $x\le_L y\Cup x=x\Cup y$, hence $x\Cup y\in F$, so that $F$ satisfies property $(P_2)$. 
Conversely, let $X$ be an implicative involutive BE algebra and assume that every $F\in \mathcal{DS}_o(X)$ 
verifies property $(P_2)$. Let $x\in F$, $y\in X$ and consider $F=[x,1]\in \mathcal{DS}_o(X)$. 
By $(P_2)$, $x\Cup y\in F$, hence $x\le_L x\Cup y$. 
Since $(P_2)$ implies $(P_1)$, then by $(1)$, $X$ is an IOML, and by Theorem \ref{ioml-90-10}, $X$ satisfies 
property $(QW_1)$. 
Since by Theorem \ref{dioml-05-50}, $x\le_L x\Cup y$ implies $x\Cup y=(x\Cup y)\Cup x$, using property $(QW_1)$, we get: \\
$\hspace*{2.00cm}$ $x\Cup y=((x\Cup y)\ra x)\ra x=(x^*\ra (x\Cup y)^*)\ra x$ \\
$\hspace*{3.00cm}$ $=(x^*\ra (x^*\Cap y^*))\ra x=(x^*\ra y^*)\ra x$ \\
$\hspace*{3.00cm}$ $=(y\ra x)\ra x=y\Cup x$. \\
Hence $X$ is an implicative-Boolean algebra.  
\end{proof}

\begin{theorem} \label{dsioml-50} Let $X$ be an implicative involutive BE algebra. Then: \\
$(1)$ $X$ is an implicative-orthomodular lattice iff for any $x,y\in X$, $y\nleqslant_L x$, there exists 
$F\in \mathcal{DS}_o(X)$ satisfying $(P_1)$ such that $y\in F$, $x\notin F;$ \\
$(2)$ $X$ is an implicative-Boolean algebra iff for any $x,y\in X$, $y\nleqslant_L x$, there exists 
$F\in \mathcal{DS}_o(X)$ satisfying $(P_2)$ such that $y\in F$, $x\notin F$. 
\end{theorem}
\begin{proof}
$(1)$ If $X$ is an IOML, then for any $x,y\in X$, $y\nleqslant_L x$, the interval $[y,1]$ satisfies the 
required conditions. 
Conversely, let $x,y\in X$, $y\nleqslant_L x$, and let $F_y$ be the smallest order deductive system with 
property $(P_1)$ containing $y$. 
It follows that for every $F\in \mathcal{DS}_o(X)$ with property $(P_1)$ containing $y$, we have $F_y\subseteq F$. 
From $y\nleqslant_L x$ we have $x\notin F_y$, hence $F_y=[y,1]$. 
Since $F_y$ has property $(P_1)$, similarly as in the proof of Theorem \ref{dsioml-40}$(1)$, $X$ is an IOML. \\
$(2)$ If $X$ is an implicative-Boolean algebra, then for any $x,y\in X$, $y\nleqslant_L x$, the interval $[y,1]$ satisfies the required conditions. 
Conversely, similarly as in the case $(1)$, $[y,1]$ satisfies condition $(P_2)$. 
Following the proof of Theorem \ref{dsioml-40}$(2)$, $X$ is an implicative-Boolean algebra. 
\end{proof}

$\vspace*{1mm}$

\section{Types of states on implicative-orthomodular lattices} 

In this section, we study in the case of implicative-orthomodular lattices certain notions and results presented by 
S. Pulmannov\'a in \cite{Pulm1} for orthomodular lattices. 
We define various type of states on implicative involutive BE algebras (Jauch-Piron state, (P)-state, (B)-state, 
subadditive state, valuation), and we investigate the relationships between these states. 
We prove that, in the case of implicative-orthomodular lattices, the (B)-states, subadditive states and valuations coincide. 

\begin{definition} \label{sioml-10} \emph{
A \emph{state} on an implicative involutive BE algebra $X$ is a map $s:X\longrightarrow [0,1]$ satisfying the 
following conditions: \\
$(S_1)$ $s(1)=1;$ \\
$(S_2)$ $s(x\ra y)=s(x^*)+s(y)$, whenever $x,y\in X$ such that $y\le_L x$. 
}
\end{definition}

Denote by $\mathcal{S}(X)$ the set of all states on $X$. 

\begin{lemma} \label{sioml-20} Let $X$ be an implicative involutive BE algebra, and let $s\in \mathcal{S}(X)$. 
The following hold, for $x,y\in X$: \\
$(1)$ $s(0)=0;$ \\
$(2)$ $s(x^*)=1-s(x);$ \\
$(3)$ $y\le_L x$ implies $s(y)\le s(x)$. 
\end{lemma}
\begin{proof}
$(1)$ Take $y:=0$ in $(S_2)$. \\
$(2)$ Take $y:=x$ in $(S_2)$. \\
$(3)$ By $(S_2)$, $s(y)-s(x)=s(x\ra y)-1\le 0$, hence $s(y)\le s(x)$. 
\end{proof}

We define various types of states on implicative involutive BE algebras. 

\begin{definition} \label{sioml-30} \emph{ 
Let $X$ be an implicative involutive BE algebra, and let $s\in \mathcal{S}(X)$. Then $s$ is called: \\
$(T_1)$ a \emph{Jauch-Piron state}, if $x,y\in X$ such that $s(x)=1$ and $s(y)=0$ implies $s(x\ra y)=0;$ \\
$(T_2)$ a \emph{(P)-state}, if $x,y\in X$ such that $s(x)=1$ implies $s(x\ra y)=s(y);$ \\
$(T_3)$ a \emph{(B)-state}, if $x,y\in X$ such that $x\le y^*$ implies $s(x^*\ra y)=s(x)+s(y);$ \\
$(T_4)$ \emph{subadditive}, if $x,y\in X$ implies $s(x\ra y)\le s(x^*)+s(y);$ \\
$(T_5)$ a \emph{valuation}, if $x,y\in X$ implies $s(x\ra y)-s(y\ra x)=s(y)-s(x)$.  
}
\end{definition}

\begin{proposition} \label{sioml-40} Let $X$ be an implicative involutive BE algebra, and let $s\in \mathcal{S}(X)$.  Then the following hold: \\
$(T_5)\Rightarrow (T_4)\Rightarrow (T_2)\Rightarrow (T_1)$, $(T_5)\Rightarrow (T_3)$. 
\end{proposition}
\begin{proof}  Let $s\in \mathcal{S}(X)$. \\
$(T_5)\Rightarrow (T_4)$: Since $s$ is a valuation and $1-s(y\ra x)\ge 0$, we have $s(x\ra y)+1-s(y\ra x)=1-s(x)+s(y)$, 
so that $s(x\ra y)\le s(x^*)+s(y)$. Hence $s$ is subadditive. \\
$(T_4)\Rightarrow (T_2)$: Let $x,y\in X$ such that $s(x)=1$. Since $y\le_L x\ra y$ and $s$ is a subadditive state, we have 
$s((x\ra y)\ra y)=1-s(x\ra y)+s(y)\ge 1-1+s(x)-s(y)+s(y)=s(x)=1$. 
Hence $s((x\ra y)\ra y)=1$, that is $1-s(x\ra y)+s(y)=1$, so that $s(x\ra y)=s(y)$. Thus $s$ is a (P)-state. \\
$(T_2)\Rightarrow (T_1)$: It is obvious. \\
$(T_5)\Rightarrow (T_3)$: Let $x,y\in X$ such that $x\le y^*$. Since $s$ is a valuation and $x\le y^*$ 
implies $y\ra x^*=x\ra y^*=1$, from $s(x^*\ra y)-s(y\ra x^*)=s(y)-s(x^*)$, we get $s(x^*\ra y)=s(x)+s(y)$. 
Hence $s$ is a (B)-state. 
\end{proof}

\begin{proposition} \label{sioml-50} Let $X$ be an implicative-orthomodular lattice and let $s\in \mathcal{S}(X)$. 
Then $(T_4)\Rightarrow (T_3)$. 
\end{proposition}
\begin{proof}
Let $s\in \mathcal{S}(X)$ satisfying $(T_4)$, and let $x,y\in X$ such that $x\le y^*$. \\
From $x\le_Q x^*\ra y$, $y\le_Q x^*\ra y$, we get $(x^*\ra y)^*\le_Q x^*$ and $(x^*\ra y)^*\le_Q y^*$. 
According to Proposition \ref{cioml-20}, $(x^*\ra y)\mathcal{C} x^*$ and $(x^*\ra y)\mathcal{C} y^*$. 
It follows that the triple $(x^*,y^*,x^*\ra y)$ is distributive, and applying the Foulis-Holland theorem for 
$X:=x^*$, $Y:=y^*$, $Z:=x^*\ra y$, we have $((X^*\ra Y)\ra Z^*)^*=(X\ra Z^*)\ra (Y\ra Z^*)^*$, that is: 
$((x\ra y^*)\ra (x^*\ra y)^*)^*=(x^*\ra (x^*\ra y)^*)\ra (y^*\ra (x^*\ra y)^*)^*$. 
Since $x\ra y^*=1$, we get $x^*\ra y=(x^*\ra (x^*\ra y)^*)\ra (y^*\ra (x^*\ra y)^*)^*$. 
Using the subadditivity property of $s$, we get: \\
$\hspace*{2cm}$ $s(x^*\ra y)\le s(x^*\ra (x^*\ra y)^*)^*+s(y^*\ra (x^*\ra y)^*)^*$, so that \\
$\hspace*{2cm}$ $s(x^*\ra y)\le 1-s(x^*\ra (x^*\ra y)^*)+1-s(y^*\ra (x^*\ra y)^*)$. \\
Taking into consideration that $(x^*\ra y)^*\le_Q x^*$ and $(x^*\ra y)^*\le_Q y^*$, we get: \\
$\hspace*{2cm}$ $s(x^*\ra y)\le 1-(s(x)+1-s(x^*\ra y))+1-(s(y)+1-s(x^*\ra y))$, \\
so that $s(x)+s(y)\le s(x^*\ra y)$. On the other hand, by subadditivity, $s(x^*\ra y)\le s(x)+ s(y)$. 
Hence $s(x^*\ra y)=s(x)+s(y)$, that is $(T_3)$. 
\end{proof}

\begin{proposition} \label{sioml-60} Let $X$ be an implicative-orthomodular lattice and let $s\in \mathcal{S}(X)$. 
Then $(T_3)\Rightarrow (T_5)$. 
\end{proposition}
\begin{proof}
Let $s\in \mathcal{S}(X)$ satisfying $(T_3)$ and let $x,y\in X$. 
Denote $X:=(x\ra (x\ra y)^*)^*$, $Y:=(y^*\ra (x\ra y)^*)^*$, and we have: \\
$\hspace*{2.00cm}$ $X\ra Y^*=(x\ra (x\ra y)^*)^*\ra (y^*\ra (x\ra y)^*)$ \\
$\hspace*{3.50cm}$ $=((x\ra y)\ra x^*)^*\ra ((x\ra y)\ra y)$ \\
$\hspace*{3.50cm}$ $=(x\ra y)\ra (x\ra ((x\ra y)\ra y))$ (by Lemma \ref{qbe-10}$(7)$) \\
$\hspace*{3.50cm}$ $=(x\ra y)\ra ((x\ra y)\ra (x\ra y))=1$. \\
Hence $X\le Y^*$, and by $(T_3)$, $s(X^*\ra Y)=s(X)+s(Y)$. \\
On the other hand, since $x^*\le_Q x\ra y$ implies $(x\ra y)^*\le_Q x$, we have $x\mathcal{C} (x\ra y)^*$, so that 
$x\mathcal{C} (x\ra y)$. Similarly $y\le_Q x\ra y$ implies $y\mathcal{C} (x\ra y)$, hence $y^*\mathcal{C} (x\ra y)$. 
Applying the Foulis-Holland theorem for the distributive triple $(x,y^*,x\ra y)$, we get: \\
$\hspace*{2.00cm}$ $X^*\ra Y=(x\ra (x\ra y)^*)\ra (y^*\ra (x\ra y)^*)^*=((x^*\ra y^*)\ra (x\ra y)^*)^*$ \\
$\hspace*{3.50cm}$ $=((y\ra x)\ra (x\ra y)^*)^*$. \\
It follows that $s((y\ra x)\ra (x\ra y)^*)^*=s(x\ra (x\ra y)^*)^*+s(y^*\ra (x\ra y)^*)^*$. \\
Obviously, $(x\ra y)^*\le_Q x$ and $y\le_Q x\ra y$ implies $(x\ra y)^*\le_Q y^*\le_Q x^*\ra y^*=y\ra x$. 
Then we get successively: \\ 
$\hspace*{0.50cm}$ $1-s((y\ra x)\ra (x\ra y)^*)=1-s(x\ra (x\ra y)^*)+1-s(y^*\ra (x\ra y)^*)$, \\
$\hspace*{0.50cm}$ $1-(s(y\ra x)^*+s(x\ra y)^*)=1-(s(x^*)+s(x\ra y)^*)+1-(s(y)+s(x\ra y)^*)$, \\
$\hspace*{0.50cm}$ $1-1+s(y\ra x)-1+s(x\ra y)=1-s(x^*)-1+s(x\ra y)+1-s(y)-1+s(x\ra y)$, \\
$\hspace*{0.50cm}$ $s(y\ra x)-s(x\ra y)=1-s(x^*)-s(y)$, \\
$\hspace*{0.50cm}$ $s(x\ra y)-s(y\ra x)=s(y)-s(x)$, \\
that is $(T_5)$. 
\end{proof}

\begin{theorem} \label{sioml-70} Let $X$ be an implicative-orthomodular lattice, and let $s\in \mathcal{S}(X)$. 
Then $(T_5)\Leftrightarrow (T_4)\Leftrightarrow (T_3)$ and $(T_3)\Rightarrow (T_2)\Rightarrow (T_1)$. 
\end{theorem}
\begin{proof}
It follows from Propositions \ref{sioml-40}, \ref{sioml-50}, \ref{sioml-60}. 
\end{proof}

A state $s$ on an implicative involutive BE algebra is called an \emph{$\{0,1\}$-state} if $s(x)\in \{0,1\}$, 
for all $x\in X$. 

\begin{proposition} \label{sioml-80} Let $X$ be an implicative involutive BE algebra, and let $s$ be a $\{0,1\}$-state 
on $X$. Then $s$ is a valuation if and only if it is a Jauch-Piron state. 
\end{proposition}
\begin{proof}
If $s$ is a valuation, by Proposition \ref{sioml-40}, $(T_5)\Rightarrow (T_1)$, so that $s$ is a Jauch-Piron state. 
Conversely, let $s$ be a Jauch-Piron state, and let $x,y\in X$. We have the following cases: \\
$(1)$ $s(x)=1$, $s(y)=0$. \\
Since $s$ is a Jauch-Piron state, we have $s(x\ra y)=s(y)=0$. 
Moreover, $x\le_L y\ra x$ implies $1=s(x)\le s(y\ra x)$, so that $s(y\ra x)=1$. 
It follows that $s(x\ra y)-s(y\ra x)=s(y)-s(x)=-1$. \\
$(2)$ $s(x)=0$, $s(y)=1$. \\
It follows that $s(y\ra x)=s(x)=0$. Since $y\le_L x\ra y$ implies $1=s(y)\le s(x\ra y)$, we get $s(x\ra y)=1$. 
Hence $s(x\ra y)-s(y\ra x)=s(y)-s(x)=1$. \\
$(3)$ $s(x)=s(y)=1$. \\
Similarly as in $(1)$ and $(2)$, we get $s(x\ra y)=s(y\ra x)=1$. 
Thus $s(x\ra y)-s(y\ra x)=s(y)-s(x)=0$. \\ 
$(4)$ $s(x)=s(y)=0$. \\
From $x^*\le_L x\ra y$, we have $s(x^*)\le s(x\ra y)$, so that $1=1-s(x)=s(x^*)\le s(x\ra y$. 
Hence $s(x\ra y)=1$ and similarly, $y^*\le_L y\ra x$ implies $s(y\ra x)=1$. 
We get $s(x\ra y)-s(y\ra x)=s(y)-s(x)=0$. \\
We conclude that $s$ is a valuation on $X$. 
\end{proof}

\begin{definition} \label{sioml-90} \emph{
Let $X$ be an implicative involutive BE algebra, and let $s\in \mathcal{S}(X)$. 
The set $\Ker(s)=\{x\in X\mid s(x)=1\}$ is called the \emph{kernel} of $s$. 
}
\end{definition}

\begin{proposition} \label{sioml-100} Let $X$ be an implicative involutive BE algebra, and let $s\in \mathcal{S}(X)$. 
The following hold: \\
$(1)$ $\Ker(s)\in \mathcal{DS}_o(X);$ \\
$(2)$ $\Ker(s)\in \mathcal{DS}_q(X)$ iff $s$ is a Jauch-Piron state; \\
$(3)$ $\Ker(s)\in \mathcal{DS}_p(X)$ iff $s$ is a (P)-state. 
\end{proposition}
\begin{proof}
$(1)$ Let $x\in \Ker(s)$ and $y\in X$ such that $x\le_L y$. 
We have: $1=s(x)\le s(y)$, so that $s(y)=1$. It follows that $y\in \Ker(s)$, that is $\Ker(s)$ satisfies $(F_1)$. 
Hence $\Ker(s)\in \mathcal{DS}_o(X)$. \\ 
$(2)$ Suppose $\Ker(s)\in \mathcal{DS}_q(X)$, and let $x,y\in X$ such that $s(x)=1$ and $s(y)=0$. 
Then $s(y^*)=1$, that is $x, y^*\in \Ker(s)$. 
Using $(F_2)$, we get $(x\ra y)^*=(x\ra (y^*)^*)^*\in \Ker(s)$.
It follows that $s(x\ra y)^*=1$, that is $s(x\ra y)=0$. Hence $s$ is a Jauch-Piron state. \\
Conversely, assume that $s$ is a Jauch-Piron state, and let $x,y\in \Ker(s)$, that is $s(x)=s(y)=1$ and $s(y^*)=0$. 
Since $s$ is a Jauch-Piron state, we have $s(x\ra y^*)=0$, so that $s(x\ra y^*)^*=1$. 
It follows that $(x\ra y^*)^*\in \Ker(s)$, that is condition $(F_2)$ is verified. 
Moreover, by $(1)$, $\Ker(s)$ verifies $(F_1)$, hence $\Ker(s)\in \mathcal{DS}_q(X)$. \\
$(3)$ Suppose $\Ker(s)\in \mathcal{DS}_q(X)$, and let $x,y\in X$ such that $s(x)=1$, that is $x\in \Ker(s)$. 
From $x\in \Ker(s)$ and $y\in X$ we get $x\Cup y\in \Ker(s)$, hence $s((x\ra y)\ra y)=1$. 
Since $y\le_L x\ra y$, we get $s(x\ra y)^*+s(y)=1$, so that $s(x\ra y)=s(y)$, thus $s$ is a (P)-state. \\
Conversely, if $s$ is a (P)-state on $X$, then by Proposition \ref{sioml-40}, $s$ is a Jauch-Piron state,  
and applying $(2)$, $\Ker(s)\in \mathcal{DS}_q(X)$. 
Let $x\in \Ker(s)$ and $y\in X$, hence $s(x)=1$. 
Since $s$ is a (P)-state, we have $s(x\ra y)=s(y)$. It follows that 
$s(x\Cup y)=s((x\ra y)\ra y)=s(x\ra y)^*+s(y)=1-s(x\ra y)+s(y)=1$. 
Hence $x\Cup y\in \Ker(s)$, so that condition $(F_3)$ is verified and $\Ker(s)\in \mathcal{DS}_p(X)$.  
\end{proof}
          
\begin{corollary} \label{sioml-110} Let $X$ be an implicative involutive BE algebra, and let $s\in \mathcal{S}(X)$. 
The following hold: \\
$(1)$ if $x,y\in X$ such that $x\le_L y$ and $x\in \Ker(s)$, then $y, x\ra y, y\ra x\in \Ker(s);$ \\    
$(2)$ if $x\in \Ker(s)$ and $y\in X$ such that $y\le_L x$, then $x\Cup y\in \Ker(s);$ \\
$(3)$ if $s$ is a (P)-state, $x\in \Ker(s)$ and $y\in X$, then $x\Cup y\in \Ker(s)$. 
\end{corollary}
\begin{proof}
$(1)$ Since $\Ker(s)\in \mathcal{DS}_o(X)$, $x\le_L y$ implies $y\in \Ker(s)$. 
Moreover, $x\le_L y\ra x$ and $y\le_L x\ra y$ imply $y\ra x, x\ra y\in \Ker(s)$. \\
$(2)$ From $x\in \Ker(s)$, $y\le_L x$ and $y\le_L x\ra y$, we get 
$s(x\Cup y)=s((x\ra y)\ra y)=s(x\ra y)^*+s(y)=1-s(x\ra y)+s(y)=1-s(x^*)-s(y)+s(y)=s(x)=1$, 
hence $x\Cup y\in \Ker(s)$. \\
$(3)$ By Proposition \ref{sioml-100}$(3)$, $s$ is a (P)-state iff $\Ker(s)\in \mathcal{DS}_p(X)$, 
thus $x\Cup y\in \Ker(s)$.  
\end{proof}

\begin{proposition} \label{sioml-120} Let $X$ be an implicative involutive BE algebra, and let $s\in \mathcal{S}(X)$. 
The following hold: \\
$(1)$ $\Ker(s)$ satisfies property $(P_1);$ \\
$(2)$ if $s$ is a (P)-state, then $\Ker(s)$ satisfies property $(P_2)$. 
\end{proposition}
\begin{proof}
$(1)$ By Proposition \ref{sioml-100}$(1)$, $\Ker(s)\in \mathcal{DS}_o(X)$. 
Let $x\in \Ker(s)$, $y\in X$, $y\le_L x$. 
Then $s(x\Cup y)=s((x\ra y)\ra y)=s(x\ra y)^*+s(y)=1-s(x\ra y)+s(y)=1-(s(x^*)+s(y))+s(y)=s(x)=1$, 
that is $x\Cup y\in \Ker(s)$. Hence $\Ker(s)$ satisfies property $(P_1)$. \\
$(2)$ Let $x\in \Ker(s)$ and $y\in X$. Since $s(x)=1$ and $s$ is a (P)-state, we have $s(x\ra y)=s(y)$.   
Hence $s(x\Cup y)=s((x\ra y)\ra y)=s(x\ra y)^*+s(y)=1-s(x\ra y)+s(y)=1$, so that $x\Cup y\in \Ker(s)$. 
It follows that $\Ker(s)$ satisfies property $(P_2)$. 
\end{proof}

$\vspace*{1mm}$

\section{A classification of states on implicative-orthomodular lattices} 

We define the unital, full and rich sets of states on implicative involutive BE algebras and we investigate 
certain properties of these sets of states. 
If an implicative involutive BE algebra $X$ possesses a rich or a full set of states, we prove that $X$ is an 
implicative-orthomodular lattice. 
In the case when $X$ possesses a rich set of (P)-states or a full set of valuations, then $X$ is an implicative-Boolean algebra. 

\begin{definition} \label{csioml-10} \emph{
Let $X$ be an implicative involutive BE algebra, and let $S\subseteq \mathcal{S}(X)$. We say that $S$ is: \\
$(1)$ \emph{unital}, if for any $x\in X$, $x\ne 0$, there exists $s\in S$ such that $s(x)=1;$ \\
$(2)$ \emph{full}, if for any $x,y\in X$, $x\nleqslant_L y$, there exists $s\in S$ such that $s(x)\nleq s(y);$ \\
$(3)$ \emph{rich}, if for any $x,y\in X$, $x\nleqslant_L y$, there exists $s\in S$ such that $s(x)=1$ and $s(y)\ne 1$.  
}
\end{definition}

\begin{remarks} \label{csioml-20} Let $X$ be an implicative involutive BE algebra and let 
$S\subseteq \mathcal{S}(X)$. \\
$(1)$ One can easily check that a rich set $S$ of states is both unital and full. \\
Indeed, let $S$ be a full set of states, and let $x\ne 0$. Then $x\nleqslant_L 0$, so that there exists $s\in S$ 
such that $s(x)=1$ and $s(0)=0\ne 1$, hence $S$ is a unital set of states.  
Obviously, by definition, a rich set $S$ is also a full set of states. \\
$(2)$ If $S$ is full and $s(x)\le s(y)$ for any $s\in S$, then $x\le_L y$. \\ 
Indeed, if $x\nleqslant_L y$, there exists $s\in S$ such that $s(x)\nleqslant_L s(y)$, a contradiction. \\
$(3)$ If $S$ is full and $s(x)=s(y)$ for any $s\in S$, then $x=y$. 
\end{remarks}

\begin{proposition} \label{csioml-30} Let $X$ be an implicative involutive BE algebra. A unital set of Jauch-Piron 
states on $X$ rich. 
\end{proposition}
\begin{proof}
Let $S$ be a unital set of Jauch-Piron states on $X$, and let $x,y\in X$ such that $x\nleqslant_L y$, that is 
$x\ne (x\ra y^*)^*$. 
Since $(x\ra y^*)^*\le_L x$ and $x\ne (x\ra y^*)^*$, it follows that $(x\ra y^*)^*\lneq_L x$. 
Hence $x\ra (x\ra y^*)^*\ne 1$, so that $(x\ra (x\ra y^*)^*)^*\ne 0$. 
Since $S$ is unital, there exists $s\in S$ such that $s(x\ra (x\ra y^*)^*)^*=1$, hence $s(x\ra (x\ra y^*)^*)=0$. 
It follows that $s(x^*)+s(x\ra y^*)^*=0$, that is $s(x^*)=s(x\ra y^*)^*=0$. 
Thus $s(x)=1$ and $s(x\ra y^*)=1$. 
We have $s(y^*)\ne 0$, otherwise, since $s$ is a Jauch-Piron state, $s(x\ra y^*)=0$. Hence $s(y)\ne 1$. 
We conclude that $S$ is a rich set of states. 
\end{proof}

\begin{proposition} \label{csioml-40} Let $X$ be an implicative involutive BE algebra. \\
$(1)$ if $X$ possesses a rich set of states, then $X$ is an implicative-orthomodular lattice; \\
$(2)$ if $X$ possesses a rich set of (P)-states, then $X$ is an implicative-Boolean algebra. 
\end{proposition}
\begin{proof}
$(1)$ Let $S\subseteq \mathcal{S}(X)$ be a rich set of states on $X$. 
By Proposition \ref{sioml-100}$(1)$, $\Ker(s)\in \mathcal{DS}_o(X)$ for any $s\in S$, and by 
Proposition \ref{sioml-120}$(1)$, $\Ker(s)$ satisfies property $(P_1)$. 
Let $x,y\in X$, $y\nleqslant_L x$. Since $S$ is rich, there exists $s\in S$ such that $s(y)=1$, $s(x)\ne 1$, 
so that $y\in \Ker(s)$, $x\notin \Ker(s)$. By Theorem \ref{dsioml-50}$(1)$, $X$ is an IOML. \\
$(2)$ Let $S\subseteq \mathcal{S}(X)$ be a rich set of (P)-states on $X$. 
By Proposition \ref{sioml-100}$(3)$, $\Ker(s)\in \mathcal{DS}_p(X)$, so that $\Ker(s)\in \mathcal{DS}_o(X)$.  
Let $x,y\in X$, $y\nleqslant_L x$. Since $S$ is rich, there exists $s\in S$ such that $s(y)=1$, $s(x)\ne 1$, 
so that $y\in \Ker(s)$, $x\notin \Ker(s)$. 
By Proposition \ref{sioml-120}$(2)$, $\Ker(s)$ satisfies property $(P_2)$.
It follows that $\Ker(s)$ satisfies conditions from Theorem \ref{dsioml-50}$(2)$, hence $X$ is an 
implicative-Boolean algebra.
\end{proof}

\begin{proposition} \label{csioml-50} Let $X$ be an implicative involutive BE algebra. \\
$(1)$ if $X$ possesses a full set of states, then $X$ is an implicative-orthomodular lattice; \\
$(2)$ if $X$ possesses a full set of valuations, then $X$ is an implicative-Boolean algebra. 
\end{proposition}
\begin{proof}
$(1)$ Let $S\subseteq \mathcal{S}(X)$ be a full set of states on $X$, let $s\in S$, and let $x\le_L y$. 
Since $x\le_L y\ra x$, we have: \\ 
$\hspace*{2.00cm}$ $s(y\Cup x)=s((y\ra x)\ra x)=s(y\ra x)^*+s(x)=1-s(y\ra x)+s(y)$ \\
$\hspace*{3.50cm}$ $=1-s(y^*)-s(x)+s(x)=s(y)$. \\
Hence $s(y\Cup x)=s(y)$ for any $x\le_L y$ and all $s\in S$. 
Since $S$ is full, then by Remarks \ref{csioml-20}$(3)$, we get $y\Cup x=y$ for any $x\le_L y$. 
Hence, by Theorem \ref{dioml-05-50}, $X$ is an IOML. \\
$(2)$ Let $S\subseteq \mathcal{S}(X)$ be a full set of valuations on $X$, let $s\in S$, and let $x,y\in X$. 
Since by Proposition \ref{ioml-50}$(8)$, $(x\Cup y)\ra (y\Cup x)=(y\Cup x)\ra (x\Cup y)=1$, we have 
$0=s((x\Cup y)\ra (y\Cup x))-s((y\Cup x)\ra (x\Cup y))=s(y\Cup x)-s(x\Cup y)$,  
hence $s(x\Cup y)=s(y\Cup x)$ for all $s\in S$. 
Since $S$ is full, by Remarks \ref{csioml-20}$(3)$, it follows that $x\Cup y=y\Cup x$ for all $x,y\in X$, 
thus $X$ is an implicative-Boolean algebra. 
\end{proof}


$\vspace*{1mm}$


\begin{thebibliography}{99}


\bibitem{Beran} L. Beran, \emph{Orthomodular Lattices: Algebraic Approach. Mathematics and its Applications}, 
Springer, Netherland, 1985.

\bibitem{Birk1} G. Birkhoff, J. von Neumann, \emph{The logic of quantum mechanics}, Ann. Math. {\bf 37}(1936), 823--834. 

\bibitem{Birk2} G. Birkhoff, \emph{Lattice theory}, Amer. Math. Soc. Colloq. Publ., Amer. Math. Soc., 1961. 



\bibitem{Ciu83} L.C. Ciungu, \emph{Implicative-orthomodular lattices}, arxiv.org:2401.12845. 

\bibitem{Ciu84} L.C. Ciungu, \emph{Foulis-Holland theorem for implicative-orthomodular lattices}, arxiv.org:2403.16762. 

\bibitem{DvPu} A. Dvure\v censkij, S. Pulmannov\'a, \emph{New trends in Quantum Structures}, Kluwer Academic Publishers, Dordrecht, Ister Science, Bratislava, 2000.





\bibitem{Husimi} K. Husimi, \emph{Studies on the foundation of quantum mechanics. I.}, Proc. Phys. Math. Soc. Jpn. 
{\bf 19}(1937), 766--789.

\bibitem{Ior32} A. Iorgulescu, \emph{On quantum-MV algebras - Part II: Orthomodular lattices, softlattices and widelattices}, Trans. Fuzzy Sets Syst. {\bf 1}(2022), 1--41.

\bibitem{Ior35} A. Iorgulescu, \emph{BCK algebras versus m-BCK algebras}. Foundations, Studies in Logic, 
Vol. 96, 2022. 

\bibitem{Kim1} H.S. Kim, Y.H. Kim, \emph{On BE-algebras}, Sci. Math. Jpn. {\bf 66}(2007), 113--116.

\bibitem{PadRud} R. Padmanabhan, S. Rudeanu, \emph{Axioms for lattices and Boolean algebras}, World Scientific, Singapore, 2008.


\bibitem{Pulm1} S. Pulmannov\'a, \emph{A remark on states on orthomodular lattices}, Tatra Mountains Math. Publ. 
{\bf 2}(1993), 209--219.

\bibitem{Varad1} V.S. Varadarajan, \emph{Geometry of Quantum Theory}, Vol. 1, D. Van Nostrand, Princeton, New Jersey, 1968.

\end{thebibliography}
\end{document}